\documentclass{IEEEtran}

\scrollmode

\usepackage{amsmath}
\usepackage{amsthm}
\usepackage{amssymb}
\usepackage{srcltx}
\usepackage{hyperref}
\usepackage{tikz}

\newtheorem{lemma}{Lemma}[section]
\newtheorem{theorem}[lemma]{Theorem}
\newtheorem{corollary}[lemma]{Corollary}
\newtheorem{proposition}[lemma]{Proposition}

\theoremstyle{definition}
\newtheorem{example}[lemma]{Example}
\newtheorem{definition}[lemma]{Definition}
\newtheorem{remark}[lemma]{Remark}

\numberwithin{equation}{section}

\newcommand{\leaveout}[1]{}

\makeatletter
\renewcommand{\p@enumii}{}
\makeatother

\newcommand{\DD}[1]{\mathbin{\frac{\rm d }{{\rm d }#1}}}
\newcommand{\ddt}{\DD t}

\newcommand{\ud}{\,{\mathrm d}}

\newcommand\R{{\mathbb R}}
\newcommand\C{{\mathbb C}}

\newcommand\K{{\mathbb K}}
\newcommand\F{{\mathbb F}}

\newcommand\rplus{{\R_{+}}}
\newcommand\cplus{{\C_{+}}}

\newcommand\T{{\mathbb T}}

\newcommand{\Escr}{\mathcal E}
\newcommand{\Fscr}{\mathcal F}

\newcommand{\Lscr}{\mathcal L}

\newcommand{\Uscr}{\mathcal U}
\newcommand{\Xscr}{\mathcal X}
\newcommand{\Yscr}{\mathcal Y}
\newcommand{\Zscr}{\mathcal Z}
\newcommand{\Vscr}{\mathcal V}

\newcommand{\proj}{\mathbf P}

\newcommand{\AB}{{A\& B}}
\newcommand{\CD}{{C\& D}}

\newcommand{\SysNode}{\bbm{\AB \cr \CD}}
\newcommand{\SmallSysNode}{\sbm{\AB \cr \CD}}

\newcommand{\dom}[1]{\mathrm{dom}\left(#1\right)}

\newcommand{\re}[0]{\mathrm{Re}\,}

\newcommand{\Ipdp}[2]{\left\langle #1 , #2 \right\rangle}

\newcommand{\set}[1]{\left\lbrace #1 \right\rbrace}

\newcommand{\biggmid}{\biggm\vert}

\newcommand{\bi}{\begin{itemize}}
\newcommand{\ei}{\end{itemize}}
\newcommand{\be}{\begin{enumerate}}
\newcommand{\ee}{\end{enumerate}}

\newcommand{\Afrak}{\mathfrak A}
\newcommand{\Bfrak}{\mathfrak B}
\newcommand{\Cfrak}{\mathfrak C}

\newcommand{\sbm}[1]{\left[\begin{smallmatrix}#1\end{smallmatrix}\right]}

\newcommand{\bbm}[1]{\begin{bmatrix}#1\end{bmatrix}}

\makeatletter

\def\etv{& \hskip-.3em\vrule\hskip-.3em &} 
\def\smalletv{&\vrule&} 
\def\smallcrh{\vrule height0pt depth2\ex@ width0pt
\cr\noalign{\hrule}
\vrule height6.5\ex@ depth0pt width0pt}
\newbox\smallstrutbox
\setbox\smallstrutbox=\hbox{\vrule height6pt depth1.5pt width\z@}
\def\smallstrut{\relax\ifmmode\copy\smallstrutbox\else\unhcopy\smallstrutbox\fi}

\newenvironment{sysmatrix}{
\let\|=\etv
\hskip \arraycolsep
\begin{matrix}}
{\end{matrix}
\hskip \arraycolsep
}       

\newenvironment{smallsysmatrix}{\null\,\vcenter\bgroup
\let\|=\smalletv

\def\\{\smallstrut\math@cr}
\restore@math@cr\default@tag
\baselineskip\z@skip \lineskip\z@skip \lineskiplimit\lineskip
\ialign\bgroup\hfil$\m@th\scriptstyle##$\hfil&&\thickspace\hfil
$\m@th\scriptstyle##$\hfil\crcr
\crcr\noalign{\vskip -.3\ex@}%
}{\crcr\noalign{\vskip -.2\ex@}%
\crcr\egroup\egroup\,%
}

\makeatother

\begin{document}

\title{Solvability of time-varying infinite-dimensional linear port-Hamiltonian systems}

\author{Mikael Kurula

\thanks{M. Kurula is with \AA bo Akademi University Mathematics, Henriksgatan 2, 20500 \AA bo, Finland (e-mail: mkurula@abo.fi).}}

\thispagestyle{empty}

\maketitle

\begin{abstract}
Thirty years after the introduction of port-Hamiltonian systems, interest in this system class still remains high among systems and control researchers. Very recently, Jacob and Laasri obtained strong results on the solvability and well-posedness of time-varying linear port-Hamil\-to\-nian systems with boundary control and boundary observation. In this paper, we complement their results by discussing the solvability of linear, infinite-dimensional time-varying port-Hamiltonian systems not necessarily of boundary control type. The theory is illustrated on a system with a delay component in the state dynamics.
\end{abstract}

\section{Introduction}

Ever since they were introduced by Maschke and van der Schaft in \cite{MaSch92}, the study of port-Hamiltonian systems have inspired intensive research. For a nice overview of the fundaments and early history of port-Hamiltonian systems, see the introduction to \cite{vdSJel14}. In his PhD thesis \cite{GoloPhD}, Golo started the study of \emph{infinite-dimensional} Dirac structures, which describe the structure of a given port-Hamiltonian system. The methods of semigroup theory were first put to use in the study of port-Hamiltonian systems by Le Gorrec, Zwart and Maschke in \cite{GZM05}. This direction was further developed in \cite{KZSB10} and \cite{KuZw15}, the latter being the first treatment of a port-Hamiltonian system with a spatial dimension larger than one. In \cite{Skr21}, Skrepek studied boundary controlled port-Hamiltonian systems on a high-dimensional spatial domain. Also very recently, in \cite{AuLa20,JaLa21}, Augner, Jacob and Laasri  studied solvability, well-posedness and exponential stability for time-varying \emph{boundary controlled} port-Hamiltonian systems, hence starting the study of \emph{time-varying} infinite-dimensional port-Hamiltonian systems. In the aforementioned references, one can find applications to the string and wave equation, beam and plate equations, and Maxwell's equations.

In this short note, we will study the solvability of time-varying, linear but in general infinite-dimensional, port-Hamiltonian systems, which are \emph{not necessarily} of boundary control type. To the author's knowledge, this has not yet been researched. 

In \S\ref{sec:Ex}, we will illustrate the theory using a time-varying finite-dimensional system with a delay $\tau>0$ in the state, 
\begin{equation}\label{eq:delayprel}
	\left\{\begin{aligned} \dot z(t)&=(J-R)H(t)z(t)-A_1H(t-\tau)z(t-\tau)+Ee(t),\\
	 f(t)&=E^*H(t)z(t),\quad 
	 t\geq0,\quad (Hx)|_{[-\tau,0]}=w_0~\text{given}, \end{aligned}\right.
\end{equation}
where, for simplicity, $J=-J^*,R\geq0,A_1,E, H(t)>0$ are all bounded operators, with some additional properties which will be described later; $H>0$ means that $H$ is self-adjoint and coercive: for some $\gamma>0$, all $z\in\dom H$, $\Ipdp{Hz}{z}\geq \gamma\|z\|^2$.

In \S\ref{sec:Dirac}, we introduce the \emph{Dirac node}, an infinite-dimensional analogue of the connecting matrix $\sbm{A&B\\C&D}$ of a linear time-invariant system and an operator version of the so-called Dirac structure \cite{KZSB10}, so that \eqref{eq:delayprel} can be written in a form resembling
\begin{equation}\label{eq:pHimpIntro}
	\bbm{\dot x(t)\\f(t)}=\bbm{A+G(t)P(t)^{-1}&B\\B^*&0}\bbm{P(t)x(t)\\e(t)},
	\quad t\geq0.
\end{equation}
We also touch on the solvability properties of a Dirac node.

In \S\ref{sec:TVpH}, we make a precise definition of a time-varying port-Hamiltonian system, and define its \emph{external scattering representation}, giving some of its properties. In \S\ref{sec:Scatt}, we apply the theory of time-varying well-posed systems to the scattering representation, and \S\ref{sec:pHtraj} ends the theoretical part with discussion on the solvability of the \emph{impedance representation} \eqref{eq:pHimpIntro}. Our main contributions are Sections II, V and VI.

\section{Dirac nodes and system nodes}\label{sec:Dirac}

Let the \emph{state space} $\Xscr$, the \emph{space of external efforts} $\Escr$ and the \emph{space of external flows} $\Fscr$ all be (complex or real) Hilbert spaces. Assume that $\Escr$ and $\Fscr$ are dual, and let $\psi:\Fscr\to\Escr$ be the associated unitary duality map, which satisfies
$$
	\Ipdp{\psi f}{e}_\Escr=\Ipdp{f}{e}_{\Fscr,\Escr}, \quad
	f\in\Fscr,~e\in\Escr.
$$

\begin{definition}\label{def:DiracNode}
By a \emph{Dirac node} on $(\Escr,\Xscr,\Fscr)$, we mean an in general unbounded linear operator
$$
	\SysNode:\bbm{\Xscr\\\Escr}\supset
		\dom\SmallSysNode\to \bbm{\Xscr\\\Fscr},
$$
such that $\sbm{\AB\\-\psi\CD}$ is maximal dissipative on $\sbm{\Xscr\\\Escr}$ and 
\begin{equation}\label{eq:InjBeta}
	\bbm{I_\Xscr&0\\0&\beta I_\Escr}+\bbm{0\\\psi\CD} 
\end{equation}
is injective for some $\beta\in\cplus:=\set{z\in\C\mid \re z>0}$.
\end{definition}

We next give two examples which show that most of the currently known linear port-Hamiltonian systems are governed by Dirac nodes. First consider the finite-dimensional system 
$$
	\dot z(t)=Az(t)+Be(t),\quad f(t)=Cz(t)+De(t).
$$

\begin{example}
Let $\Xscr=\K^n$ and $\Escr=\Fscr=\K^m$, where $\K=\R$ or $\K=\C$; then $\psi:\Escr\to\Escr$ is the identity matrix. Moreover, $\sbm{\AB\\-\psi\CD}=\sbm{A&B\\-C&-D}$ is maximal dissipative if and only if
$$
	A+A^*\leq0,\quad B=C^*,\quad D+D^*\geq0.
$$ 
Then $I+D$ is invertible, and
$$
 \left(\bbm{I_\Xscr&0\\0&I_\Escr}+\bbm{0\\\psi\CD}\right)
 \bbm{x\\e}=\bbm{I&0\\C&I+D}\bbm{x\\e}=0
$$
implies $x=0$, so that, moreover, $(I+D)e=0\implies e=0$.
\end{example}

The next example is an abstract class of boundary control systems, which contains most of the infinite-dimensional port-Hamiltonian systems which have been studied so far. In particular, it contains port-Hamiltonian systems with a one-dimensional spatial domain \cite{GZM05,JaZwBook}, the wave equation \cite{KuZw15}, as well as Maxwell's equations and the Mindlin plate \cite{Skr21}. For more background on the example, see \cite[\S4]{KZSB10} and \cite[\S5]{MaSt07}.

\begin{example}
Let $A_0$ be a closed, densely defined and symmetric operator on $\Xscr$, and let $(\Escr,\Gamma_1,\Gamma_2)$ be a boundary triplet \cite{GoGoBook} for $A_0^*$. Letting
$$
	\Afrak:=iA_0^*,\quad \Bfrak:=\Gamma_1,\quad
	\Cfrak:=-i\Gamma_2,
$$
we get that the operator defined by
$$
\begin{aligned}
	\SysNode&:=\bbm{\Afrak\\\psi^*\Cfrak}
	\bbm{I_\Xscr\\\Bfrak}^{-1},\\
	\dom{\SmallSysNode}&:=\bbm{I_\Xscr\\\Bfrak}\dom \Afrak,
\end{aligned}
$$
is a Dirac node: $\sbm{\AB\\-\psi\CD}^*=-\sbm{\AB\\-\psi\CD}$ and it is easy to see that \eqref{eq:InjBeta} with $\beta=1$ is injective.
\end{example}

A Dirac node describes the relationship between internal dynamics and external efforts and flows via
\begin{equation}\label{eq:LTI}
	\bbm{\dot x(t)\\f(t)}=\SysNode\bbm{x(t)\\e(t)},\qquad t\geq0.
\end{equation}
We will now see that it does in general not induce a ``system'' in a very strong sense. In order for \eqref{eq:LTI} to be solvable, the connecting operator $\SmallSysNode$ should have the structure of a \emph{system node} \cite[Lemma 4.7.7]{StafBook}:

\begin{definition}
By a \emph{system node} on $(\Escr,\Xscr,\Fscr)$, we mean a \emph{closed} linear operator 
$$
	\SysNode:\bbm{\Xscr\\\Escr}\supset
		\dom\SmallSysNode\to\bbm{\Xscr\\\Fscr}
$$
with the following properties:
\begin{enumerate}
\item the operator $\AB:\sbm{\Xscr\\\Escr}\supset\dom\AB\to\Xscr$ is closed, where $\dom\AB=\dom\SmallSysNode$,
\item the \emph{main operator} $A:\Xscr\supset\dom A\to\Xscr$ defined by
$$
	Ax:=\AB\bbm{x\\0},\quad \bbm{x\\0}\in\dom\SmallSysNode,
$$
generates a $C_0$-semigroup $\T$ on $\Xscr$, and
\item for every $e\in\Escr$, there exists an $x\in\Xscr$, such that $\sbm{x\\e}\in\dom\SmallSysNode$.
\end{enumerate}
\end{definition}

These three conditions essentially boil down to ``$A$ generating a solution, while at the same time being the most unbounded part of the node''. There is unfortunately no reason to believe that this should be the case for a general Dirac node. We have the following solvability result \cite[Thm 4.6.11]{StafBook}:

\begin{theorem}\label{thm:solution}
Let $\SmallSysNode$ be a system node, let $e\in H^1_{loc}(\rplus;\Escr)$ and $\sbm{x_0\\e(0)}\in\dom\SmallSysNode$. Then there exist $x\in C^1(\rplus;\Xscr)$ and $f\in H^1_{loc}(\rplus;\Fscr)$, such that \eqref{eq:LTI} holds.
\end{theorem}

The triple $(e,x,f)$ in Thm \ref{thm:solution} is called a  \emph{classical solution} of \eqref{eq:LTI}. A system node is \emph{impedance passive} if it is a Dirac node; see Thm 4.2 and Lemma 4.3 in \cite{StafPassI}. It is \emph{well-posed} if there are $t>0$ and $K_t$, such that all classical solutions satisfy
\begin{equation}\label{eq:WP}
	\|x(t)\|^2+\int_0^t\|f(s)\|^2\ud s \leq
	K_t\left( \|x(0)\|^2+ \int_0^t\|e(s)\|^2\ud s \right);
\end{equation}
if this inequality holds with $K_t=1$, then the system node is said to be \emph{scattering passive} \cite{StafBook}.

\section{Time-varying port-Hamiltonian systems}\label{sec:TVpH}

In order to define the time-varying port-Hamiltonian system associated to a Dirac node, we introduce two functions $P,G$, defined on $\rplus$, with values being bounded operators on the state space, which have the following properties: For all $t\geq0$:
\begin{equation}\label{eq:standing2}
\begin{aligned}
&\bullet P(t)=P(t)^*\geq0, \\
&\bullet \text{$P(t)$ has an inverse in $\Lscr(\Xscr)$,
and for all $z\in \Xscr$:}\\
& \bullet P(\cdot)z\in C^2(\rplus;\Xscr) \text{ and} \\
&\bullet P(\cdot)^{-1}z,\,G(\cdot)z\in C^1(\rplus;\Xscr).
\end{aligned}
\end{equation}

Here $\Lscr(\Xscr)$ is the space of \emph{bounded} linear operators on $\Xscr$.

\begin{definition}
Let $\SmallSysNode$ be a Dirac node and let $P(\cdot)$ and $G(\cdot)$ be as above. By the associated \emph{time-varying port-Hamiltonian system}, we mean the system
\begin{equation}\label{eq:TVpH}
	\bbm{\dot x(t)\\f(t)} = \left(\SysNode
		+\bbm{G(t)P(t)^{-1}&0\\0&0}\right)
	\bbm{P(t)x(t)\\e(t)},
\end{equation}
 $t\geq0$. The associated \emph{time-dependent Hamiltonian} is $H(x,t):=\Ipdp{P(t)x}{x}_\Xscr$.
 
By a \emph{classical solution} of \eqref{eq:TVpH} on $\rplus$, we mean a triple $(e,x,f)\in H^1_{loc}(\rplus;\Escr)\times C^1(\rplus;\Xscr)\times H^1_{loc}(\rplus;\Fscr)$, such that $\sbm{P(t)x(t)\\e(t)}\in\dom\SmallSysNode$ and \eqref{eq:TVpH} holds for all $t\geq0$.
\end{definition}

The multiplicative perturbation $P(t)$ before $x(t)$ in \eqref{eq:TVpH} encodes time-varying physical parameters of the system. The additive perturbation $G(t)P(t)^{-1}$ collects time derivatives of physical parameters (which do not appear in the time invariant case), and in the case of a wave equation, it can also encode a time-varying damper inside of the spatial domain. In Thm \ref{thm:bibo} below, $G$ will be used in a perturbation argument, in order to prove an existence result for classical solutions of \eqref{eq:TVpH}.

\begin{proposition}\label{prop:Power}
Assuming \eqref{eq:standing2}, every classical solution of \eqref{eq:TVpH} is uniquely determined by the initial state $x(0)$ and the effort signal $e$. For such solutions, the power inequality
\begin{equation}\label{eq:pHpower}
\begin{aligned}
	\ddt H(x,t)&\leq 2\re\Ipdp{f(t)}{e(t)}_{\Fscr,\Escr}+\Ipdp{\dot P(t)x(t)}{x(t)}\\
	&\quad +2\re\Ipdp{P(t)x(t)}{G(t)x(t)},\quad t\geq0,
\end{aligned}
\end{equation}
holds, with equality if $\sbm{\AB\\-\psi\CD}\subset-\sbm{\AB\\-\psi\CD}^*$.
\end{proposition}
\begin{proof}
The power inequality will be motivated after Thm \ref{thm:Sigmar}. In order to obtain uniqueness, we first observe that the main operator $A$ of a Dirac node is dissipative: for $x\in\dom A$:
$$
	2\re\Ipdp{Ax}{x}=
	2\re\Ipdp{\bbm{\AB\\-\psi\CD}\bbm{x\\0}}{\bbm{x\\0}}\leq0.
$$

From \eqref{eq:standing2} and the uniform boundedness principle, it follows that $P(\cdot)$, $P(\cdot)^{-1}$, $\dot P(\cdot)$, and $G(\cdot)^*$ are all uniformly bounded on the compact subinterval $[0,T]\subset \rplus$, where $T>0$ is fixed arbitrarily. Denote a common bound for these by $U$. 

Now let $(e,x,f)$ and $(e,z,g)$ be classical solutions of \eqref{eq:TVpH} with $x(0)=z(0)$ and set $w:=x-z$. Then
$$
	\dot w(t)=\big(AP(t)+G(t)\big)w(t),\quad t\geq0,\quad w(0)=0.
$$
Next, we define the auxiliary function
$$
	v(t):=e^{-\sigma t}H(w(t),t\big)\geq0,\quad t\geq0,
$$
where $	\sigma:= 2U^3+U^2$. Then $v(0)=0$ and we next prove that $\dot v(t)\leq0$ for all $t\in[0,T]$; then $w=0$, i.e., $x=z$, and consequently, $f=g$ by \eqref{eq:TVpH}, on $\rplus$. Indeed, for $t\in [0,T]$,
$$
	\Ipdp{P(t)w(t)}{w(t)}\geq \|w(t)\|^2/\|P(t)^{-1}\|,
	\qquad\text{and then}
$$
$$
\begin{aligned}
	e^{\sigma t}\dot v(t) &= \Big(2\re\!\Ipdp{P(t)w(t)}{\dot w(t)}
	+\Ipdp{ \dot P(t)w(t)}{w(t)}\\
	&\quad~ -\sigma\Ipdp{P(t)w(t)}{w(t)}\Big) \\
	&= \big(2\re\!\Ipdp{P(t)w(t)}{AP(t)w(t)} 
	-\sigma\Ipdp{P(t)w(t)}{w(t)} \\
	&\quad~+2\re\!\Ipdp{P(t)w(t)}{G(t)w(t)}+\langle\dot P(t)w(t),w(t)\rangle\big)
\end{aligned}
$$
is non-positive due to the dissipativity of $A$.
\end{proof}

Within the port-Hamiltonian framework, the external efforts and flows can be combined into inputs and outputs in a flexible way, and we will next consider the following \emph{external scattering representation} of a port-Hamiltonian system, which has better solvability properties than the \emph{impedance representation} \eqref{eq:TVpH}: First fix some $\beta\in\cplus$ such that \eqref{eq:InjBeta} is injective. Then pick the input $u$ to the system and the output $y$ as
\begin{equation}\label{eq:diag}
\begin{aligned}
	\bbm{u\\y}:=&\,\frac{1}{\sqrt{2\re\beta}} 
	\bbm{\beta I & \psi \\ \overline\beta I&-\psi}\bbm{e\\f}
	\quad\iff\\
	\bbm{e\\f}=&\,\frac{1}{\sqrt{2\re\beta}} 
	\bbm{I & I \\ \overline\beta\psi^*&-\beta\psi^*}\bbm{u\\y}.
\end{aligned}
\end{equation}
The reader may verify that the resulting transform of \eqref{eq:TVpH} is
\begin{equation}\label{eq:TVpHdiag}
	\bbm{\dot x(t)\\y(t)} = \left(\SysNode^\times
		+\bbm{G(t)P(t)^{-1}&0\\0&0}\right)
	\bbm{P(t)x(t)\\u(t)},
\end{equation}
$t\geq0$, where
$$
\begin{aligned}
	\SysNode^\times&\!:=
	\bbm{\sqrt{2\re\beta}\AB\\\bbm{0&\overline\beta I}-\psi\CD} \\
	&\qquad \times
	\bbm{\sqrt{2\re\beta}\bbm{I&0} \\ \bbm{0&\beta I}+\psi\CD}^{-1},\\
	\dom{\SmallSysNode^\times} &:=
	\bbm{\sqrt{2\re\beta}\bbm{I&0} \\ \bbm{0&\beta I}+\psi\CD}\dom{\SmallSysNode},
\end{aligned}
$$
is the usual \emph{external Cayley transform} (or \emph{diagonal transform} \cite[(6.2)]{StafMTNS02}) of the Dirac node $\SmallSysNode$. Hence, the time-varying parameters $G(t),P(t)$ of the port-Hamiltonian system do not interact with the external Cayley transform.

\begin{theorem}\label{thm:Sdiag}
For $\SmallSysNode$ a Dirac node, $\SmallSysNode^\times$ is a scattering passive system node. There exist a Hilbert space $\Xscr_{-1}\supset \Xscr$, and operators $A_{-1}^\times\in\Lscr(\Xscr,\Xscr_{-1})$, $B^\times\in\Lscr(\Escr,\Xscr_{-1})$, $\overline C^\times:\Xscr\supset\dom{\overline C^\times}\to\Escr$ and $D^\times\in\Lscr(\Escr)$, such that
\begin{equation}\label{eq:SdiagCompat}
	\SysNode^\times=\bbm{A_{-1}^\times&B^\times \\
		\overline C^\times&D^\times}\bigg|
		_{\dom{\SmallSysNode^\times}}.
\end{equation}
\end{theorem}

The space $\Xscr_{-1}$ above is called the \emph{extrapolation space}; see \cite[\S3.6]{StafBook}. The operators $\overline C^\times$ and $D^\times$ are unfortunately not necessarily unique. For more details, see \cite[\S5.1]{StafBook}.

\begin{proof}
As $\sbm{\AB\\-\psi\CD}$ is maximal dissipative, it is closed. Then it follows from \eqref{eq:diag} that $\SmallSysNode^\times$ is closed and maximal scattering dissipative in the sense of \cite{Sta13}; hence $\SmallSysNode^\times$ is a scattering-passive system node by \cite[Thm 2.5]{Sta13}.

By \cite[Lemma 11.1.3]{StafBook}, all scattering-passive system nodes induce well-posed systems, and since we work on Hilbert spaces, $\SmallSysNode^\times$ is then compatible by \cite[Thm 5.1.12]{StafBook}, and so \eqref{eq:SdiagCompat} holds. 
\end{proof}

We will next see how the system \eqref{eq:TVpHdiag} can be associated to a \emph{time-varying well-posed system} in the sense of \cite{Kur20}.

\section{Time-varying well-posed systems}\label{sec:Scatt}

In the time-invariant case, one observes that \eqref{eq:WP} expresses that the final state and the output signal are both bounded by the inital state and the input signal, and one then defines families of bounded operators $\T,\Phi,\Psi,\F$, such that $\sbm{\T^t&\Phi^t\\\Psi^t&\F^t}$ maps $\sbm{x(0)\\e|_{[0,t]}}$ into $\sbm{x(t)\\f|_{[0,t]}}$. In the time-varying case, one gets:

\begin{definition}\label{def:WP}
A \emph{time-varying well-posed system} on $\rplus$, with Hilbert input, state and output spaces $(\Uscr,\Xscr,\Yscr)$, is a quadruple of linear operator families defined for $(t,r)\in \Delta:=\set{(s,\sigma)\in \R_+^2\mid s\geq\sigma}$, mapping
$$
\begin{aligned}
	\T(t,r)&:\Xscr\to \Xscr, \quad \F(t,r):L^2(\rplus;\Uscr)\to L^2(\rplus;\Yscr),\\
	\quad\Phi(t,r)&:L^2(\rplus;\Uscr)\to \Xscr,\quad  \Psi(t,r):\Xscr\to L^2(\rplus;\Yscr),
\end{aligned}
$$
boundedly, which have the following additional properties:
\begin{enumerate}
\item $\T$ is an evolution family on $\Xscr$ with time interval $\rplus$,
\item the other families are \emph{causal} in the sense that
$$
\begin{aligned}
	\Phi(t,r)&=\Phi(t,r)\proj_{[t,r]},\\ 
	\Psi(t,r)&=\proj_{[t,r]}\Psi(t,r)\qquad \text{and} \\
	\F(t,r)&=\proj_{[t,r]}\F(t,r) =\F(t,r)\proj_{[t,r]},
\end{aligned}
$$
for $(t,r)\in\Delta$,
\item all four families are locally uniformly bounded,
\item and they encode the linearity of the system, so that for all $t\geq s\geq r\geq0$:
\begin{equation}\label{eq:WPlin}
\begin{aligned}
	\Phi(t,r) &= \Phi(t,s)+\T(t,s)\Phi(s,r), \\
	\Psi(t,r) &= \Psi(t,s)\T(s,0)+\Psi(s,r), \\
	\F(t,r) &= \F(t,s)+\F(s,r) + \Psi(t,s)\Phi(s,r).
\end{aligned}
\end{equation}
\end{enumerate}
\end{definition}

By a \emph{mild solution} of a well-posed system on $\rplus$ with initial state $x_0\in\Xscr$ at time $0$ and input $u\in L_{loc}^2(\rplus;\Uscr)$, we mean the triple $(u,x,y)\in L_{loc}^2(\rplus;\Uscr)\times C(\rplus;\Xscr)\times L_{loc}^2(\rplus;\Yscr)$:
\begin{equation}\label{eq:mildsol}
\begin{aligned}
	x(t)&=\T(t,0)x_0+\Phi(t,0)u\qquad \text{and} \\
	\proj_{[0,t]}y&=\Psi(t,0) x_0+\F(t,0)u,\qquad t\geq0;
\end{aligned}
\end{equation}
see \cite[p.\ 282]{SchWe10}. \emph{Classical solutions} of $\sbm{\T&\Phi\\\Psi&\F}$ are mild solutions with the smoothness required from a classical solution. 

In the time-invariant case, every classical solution of the system node is a mild solution of the well-posed system, and conversely, every mild solution which has the required additional smoothness is automatically a classical solution of the associated system node \cite[Thm 4.6.11]{StafBook}. The existence of classical solutions of \eqref{eq:TVpHdiag}, and their relationship with smooth mild solutions of the associated well-posed system $\sbm{\T&\Phi\\\Psi&\F}$, are in general much less clear than in the time-invariant case. Since the standing assumptions \eqref{eq:standing2} contain some extra smoothness, this connection becomes a bit less obscure: Define
\begin{equation}\label{eq:Vtimes}
\begin{aligned}
	\Vscr^\times&:=\bigg\{\bbm{x_0\\u}\in
	\bbm{\Xscr\\H_{loc}^1(\rplus;\Escr)}\biggmid \\ 
	&\qquad\quad \bbm{P(0)x_0\\u(0)}\in\dom{\SmallSysNode^\times}\bigg\}.
\end{aligned}
\end{equation}

The following theorem shows that, under assumptions \eqref{eq:standing2}, the time-varying well-posed system, perhaps surprisingly, has the same nice solvability properties as the time-invariant well-posed system in Thm \ref{thm:solution}; also the smooth mild solutions are the same as classical solutions. The theorem applies to all scattering passive system nodes, not only to those arising as an external Cayley transform of some Dirac node. 

\begin{theorem}\label{thm:Sigmar}
Let $\Sigma^\times$ be a scattering passive system node and assume \eqref{eq:standing2}. Then there is a time-varying well-posed system $\sbm{\T&\Phi\\\Psi&\F}$ with time interval $\rplus$, such that:
\begin{enumerate}
\item For every $\sbm{x_0\\u}\in\Vscr^\times$, there is a unique classical solution $(u,x,y)$ of \eqref{eq:TVpHdiag} on $\rplus$ with $x(0)=x_0$ and the given input signal $u$. The corresponding output satisfies $y\in H_{loc}^1(\rplus;\Yscr)$, and $(u,x,y)$ is also the unique mild solution of $\sbm{\T&\Phi\\\Psi&\F}$ with $x(0)=x_0$ and the given input signal $u$. 

\item Every classical solution of \eqref{eq:TVpHdiag} on $\rplus$ satisfies 
the power inequality
\begin{equation}\label{eq:ScattPow}
\begin{aligned}
	& \ddt \Ipdp{P(t)x(t)}{x(t)}_\Xscr\leq \Ipdp{\dot P(t)x(t)}{x(t)}_\Xscr\\
	&\qquad +2\re\Ipdp{P(t)x(t)}{G(t)x(t)}_\Xscr\\
	&\qquad+\|u(t)\|_\Uscr^2-\|y(t)\|_\Yscr^2,\quad t\geq0.
\end{aligned}
\end{equation}
Equality holds in \eqref{eq:ScattPow} if $\sbm{\AB\\-\psi\CD}\subset-\sbm{\AB\\-\psi\CD}^*$.
\end{enumerate}
\end{theorem}

Thm \ref{thm:Sigmar} is a part of \cite[Thm IV.3]{Kur20}. Observe that \eqref{eq:pHpower} follows from \eqref{eq:ScattPow} combined with \eqref{eq:diag}. 

\begin{remark}
Thm \ref{thm:Sigmar} and Prop.\ \ref{prop:Power} remain true if \eqref{eq:TVpH} is replaced by a port-Hamiltonian system in ``Berlin form'',
$$
	\bbm{P(t)\dot x(t)\\f(t)} = \left(\SysNode
		+\bbm{P(t)G(t)&0\\0&0}\right)
	\bbm{x(t)\\e(t)},
$$
and the same modification is made for \eqref{eq:TVpHdiag}. Then \eqref{eq:standing2} can be relaxed to $P(\cdot)z\in C^1(\rplus;\Xscr)$; see \cite[Thm IV.1]{Kur20}.
\end{remark}

We have the following explicit formulas for the operator families $\Psi$, $\Phi$ and $\F$, whose existence is given in Thm \ref{thm:Sigmar}:

\begin{theorem}\label{thm:exprep}
Under the assumptions of Thm \ref{thm:Sigmar}, for all $x_r\in P(r)^{-1}\,\dom {A^\times}$, 
$$
	\Psi(t,r)x_r=s\mapsto \overline C^\times P(s)\,\T(s,r)x_r,\quad r\leq s\leq t.
$$

If additionally $B^\times:\Escr\to\Xscr$ is bounded and $G(\cdot)^*z\in C^1(\rplus;\Xscr)$ for all $z\in\Xscr$, then:
\begin{enumerate}
\item For all $u\in L_{loc}^2(\rplus;U)$,
$$
	\Phi(t,r)u = \int_r^t \T(t,s)B^\times u(s)\ud s,\quad (t,r)\in\Delta.
$$

\item For all $u\in H^1_{loc}(\rplus;\Escr)$,
$$
\begin{aligned}
	\big(\F(t,r)u\big)(s) &= 
	\overline C^\times P(s)\!\int_r^s \T(s,\sigma)B^\times u(\sigma)\ud \sigma \\
	&\qquad+D^\times u(s),\quad r\leq s\leq t.
\end{aligned}
$$ 
\end{enumerate}
\end{theorem}

The representation formulas in items 1) and 2) of Thm \ref{thm:exprep} look almost the same in the case where $B^\times$ is unbounded, but they require a discussion on the extrapolation space, and here we would not benefit from getting into the details of that.

\section{Existence of classical solutions for impedance representations}\label{sec:pHtraj}

We can now examine the existence of classical solutions of a port-Hamiltonian system in impedance form \eqref{eq:TVpH}, where the external effort $e$ is considered the input and the external flow $f$ is taken as the output. The question is for which combinations of initial state $x_0$ and effort signal $e$ we can expect a classical solution of \eqref{eq:TVpH}. The impedance representation is in general not well-posed even in the time-invariant setting \cite[Thm 5.1]{StafPassI}, and then, the concept ``mild solution of \eqref{eq:TVpH}`` is not defined either. The disscussion here is mostly relevant for port-Hamiltonian systems which are not of boundary control type, because Jacob and Laasri have already obtained better results for the boundary control case in \cite{JaLa21}.

First define (with the limit taken in $\Xscr\times L^2_{loc}(\rplus;\Escr)$)
\begin{equation}\label{eq:Vdef}
	\Vscr :=\lim_{t\to\sup \rplus}
	\bbm{\sqrt{2\re\beta} I&0 \\
	\Psi^\times(t,0)&I+\F^\times(t,0)\proj_{[t,0]}}	
	\Vscr^\times.
\end{equation}
Then we have the following corollary of Thm \ref{thm:Sigmar}:

\begin{corollary}\label{cor:classTrajExist}
Under the assumptions of Thm \ref{thm:Sigmar}, for every $\sbm{x_0\\e}\in \Vscr $, there is a unique classical solution $(e,x,f)$ of \eqref{eq:TVpH} on $\rplus$ with $x(0)=x_0$ and the given effort signal $e$. (The flow satisfies $f\in H_{loc}^1(\rplus;\Escr)$.)
\end{corollary}
\begin{proof}
By \eqref{eq:mildsol} and the definitions of $\Vscr^\times$ and $\Vscr $, the condition $\sbm{x_0\\e}\in \Vscr $ means that $e$ is the effort associated to some smooth mild solution $(u,x,y)$ of $\sbm{\T&\Phi\\\Psi&\F}$ with $x(0)=x_0$, via $e=(u+y)/\sqrt{2\re\beta}$. By Thm \ref{thm:Sigmar}, $(u,x,y)$ is also a classical solution of \eqref{eq:TVpHdiag}. Setting 
$$
	f:=\frac{\overline\beta\psi^*u-\beta\psi^*y}{\sqrt{2\re\beta}}\in H^1_{loc}(\rplus;\Escr),
$$ 
we then get that $(e,x,f)$ is a classical solution of \eqref{eq:TVpH}. Uniqueness was established already in Prop.\ \ref{prop:Power}.
\end{proof}

The difficulty is of course to characterize $\Vscr $. Indeed, computing $\F$ using Thm \ref{thm:exprep} requires the evolution family $\T$, and in general, one has no analytic formula for that. 

We next characterize $\Vscr $ for the important special case of \emph{distributed control and observation}, leaving the general case as an open problem.

\begin{theorem}\label{thm:bibo}
Assume that $\SmallSysNode$ is a Dirac node with bounded control, observation and feedthrough, i.e., $\Fscr=\Escr$, $\psi=I$, $\dom{\SmallSysNode}=\sbm{\dom A\\\Escr}$ and $\SmallSysNode=\sbm{A&B\\C&D}$, where $B\in\Lscr(\Escr,\Xscr)$, $C\in\Lscr(\Xscr,\Fscr)$ and $D\in\Lscr(\Escr)$. Then $\sbm{A&B\\C&D}$ is a system node.

If additionally \eqref{eq:standing2} holds, then
\begin{equation}\label{eq:VtSimple}
	\Vscr =\Vscr^\times=
	\bbm{P(0)^{-1}\dom A\\H^1_{loc}(\rplus;\Escr)},
\end{equation}
i.e., for all $x_0\in P(0)^{-1}\dom A$ and $e\in H^1_{loc}(\rplus;\Escr)$, there exist unique $x$ and $f$, such that $(e,x,f)$ is a classical solution of \eqref{eq:TVpH} with $x(0)=x_0$. 
\end{theorem}

\begin{proof}
The dissipative operator $A$ inherits closedness and denseness of its domain from $\sbm{A&B\\-C&-D}$ which necessarily has these properties due to maximal dissipativity. Moreover, $A^*$ inherits dissipativity from $\sbm{A&B\\-C&-D}^*=\sbm{A^*&-C^*\\B^*&-D^*}$, so that $A$ generates a contraction semigroup by the Lumer-Phillips theorem. The other system node  properties are easily checked.

Due to the boundedness of $D$, it is easy to find a $\beta\in\cplus$ such that $(\beta I+D)^{-1}:\Fscr\to\Escr$ is bounded. A calculation gives the external Cayley transform $\SmallSysNode^\times$ of $\SmallSysNode$ as
$$
	\bbm{A-B(\beta I+D)^{-1}C
			& \sqrt{2\re\beta}B(\beta I+D)^{-1} \\
		-\sqrt{2\re\beta}(\beta I+D)^{-1}C
			& (\overline\beta I-D)(\beta I+D)^{-1}},
$$
with the same domain as $\sbm{A&B\\C&D}$, so that
$$
	A^\times:=A-B(\beta I+D)^{-1}C,
	\quad \dom{A^\times}=\dom A,
$$
is a bounded perturbation of $A$. From $\dom{\sbm{A&B\\C&D}}$,
$$
\begin{aligned}
	\bbm{P(0)x_0\\ u(0)}\in\dom{\SmallSysNode^\times}&
	\quad\iff \\
	x_0\in P(0)^{-1}\dom{A^\times}&,
\end{aligned}
$$
and this establishes that $\Vscr^\times$ in \eqref{eq:Vtimes} satisfies \eqref{eq:VtSimple}. 

Next we prove that $\Vscr =\Vscr^\times$. First pick $\sbm{x_0\\e}\in \Vscr $ arbitrarily; as in the proof of Cor.\ \ref{cor:classTrajExist}, there exists some classical solution $(u,x,y)$ of \eqref{eq:TVpHdiag}, with $x(0)=x_0$ and $u,y\in H^1_{loc}(\rplus;\Escr)$, and moreover $e=(u+y)/\sqrt{\re\beta}\in H^1_{loc}(\rplus;\Escr)$. Hence $\sbm{x_0\\e}\in\Vscr^\times$, and so $\Vscr \subset\Vscr^\times$.

Now let $\sbm{x_0\\e}\in \Vscr^\times$ be arbitrary and define $v:=(\beta I+D)e/\sqrt{2\re\beta}$, so that $\sbm{x_0\\v}\in\Vscr^\times$. With 
$$
	\widetilde G(t):=G(t) + B(\beta I+D)^{-1}CP(t),
	\quad t\geq0,
$$
we get $\widetilde G(\cdot)z\in C^1(\rplus;\Xscr)$, and so
$$
	\bbm{\dot x(t)\\w(t)}=\left(\SysNode^\times
		+\bbm{\widetilde G(t)P(t)^{-1}&0\\0&0}\right)
	\bbm{P(t)x(t)\\v(t)}
$$
has a unique solution $(v,x,w)$ with $x(0)=x_0$ and $v,w\in H^1_{loc}(\rplus;\Escr)$. However, this $x$ also solves
$$
\begin{aligned}
	\dot x(t)&=\big(A^\times P(t)+G(t)
		+ B(\beta I+D)^{-1}CP(t)\big)x(t) \\
	&\qquad +B^\times v(t) \\
	&= \big(AP(t)+G(t)\big) x(t)+Be(t),\quad t\geq0,
\end{aligned}
$$
and we get a classical solution $(e,x,f)$ of \eqref{eq:TVpH}, because 
$$
	f(\cdot):=CP(\cdot)x(t)+De(\cdot) \in H^1_{loc}(\rplus;\Escr).
$$
Setting $u:=(\beta e+f)/\sqrt{2\re\beta}\in H^1_{loc}(\rplus;\Escr)$, we then get $\sbm{x_0\\u}\in\Vscr^\times$, so that $\sbm{x_0\\e}\in\Vscr $. Hence, $\Vscr^\times\subset \Vscr $, and we also constructed a classical solution $(e,x,f)$ of \eqref{eq:TVpH}.
\end{proof}

If the Dirac node $\SmallSysNode$ is at the same time a scattering passive system node. Then we can apply Theorems \ref{thm:Sigmar} and \ref{thm:exprep} directly to the impedance representation \eqref{eq:TVpH} rather than to the scattering representation \eqref{eq:TVpHdiag}. We end the paper with an example which shows that such Dirac nodes actually exist.

\section{The example}\label{sec:Ex}

Denote the state space of \eqref{eq:delayprel} by $\Zscr$ and assume that this is a Hilbert space. The delay in \eqref{eq:delayprel} can be implemented by adding a delay line to the system state, leading to the system
\begin{equation}\label{eq:Delayline}
	\left\{\begin{aligned} 
	\dot w(t,\xi)&=\frac{\partial}{\partial\xi} w(t,\xi),\quad -\tau<\xi<0, ~t\geq0,\\
	 w(t,0)&=H(t)z(t),\\
	\dot z(t)&=(J-R)H(t)z(t)-A_1w(t,-\tau)+Ee(t),\\
	 f(t)&=E^*H(t)z(t);  \end{aligned}\right.
\end{equation}
observe that the delay line is filled with past $H(\cdot)z(\cdot)$, not with past $z(\cdot)$. The state space of \eqref{eq:Delayline} is $\Xscr:=\sbm{L^2(-\tau,0;\Zscr)\\\Zscr}$ which we equip with the Hilbert-space norm given by
$$
	\left\| \bbm{w\\z} \right \|_\Xscr^2:=\int_{-\tau}^0 \Ipdp{H_0w(\xi)}{w(\xi)}_\Zscr\ud\xi+\|z\|^2_\Zscr,
$$
for a bounded $H_0>0$. The space $\Fscr$ equals $\Escr$, $\psi=I_\Escr$. Set 
$$
	P(t):=\bbm{I&0\\0&H(t)}, \quad t\geq0,
$$
to get the associated Dirac node candidate
$$
	\SysNode:=\sbm{\displaystyle A&\bbm{0\\E} \\ \bbm{0&E^*}&\displaystyle 0}:
	\bbm{\Xscr\\\Escr}\supset\bbm{\dom A\\\Escr}\to\bbm{\Xscr\\\Escr},
$$
where $A:\Xscr\supset\dom A\to\Xscr$ is defined as
$$
\begin{aligned}
	A&:=\bbm{\frac{\partial}{\partial \xi}&0\\-A_1\epsilon_{-\tau}&J-R},\\
	\dom A&:=\set{\bbm{w\\z}\in \bbm{H^1(-\tau,0;\Zscr)\\\Zscr}\biggmid \epsilon_0 w=z},
\end{aligned}
$$
where $\epsilon_\xi$ denotes point evaluation at $\xi\in[-\tau,0]$.

\begin{theorem}\label{thm:DiracNode}
The operator $\SmallSysNode$ constructed above is a Dirac node if and only if
$$
	M:=\bbm{H_0&A_1^*\\A_1&2R-H_0}\geq0.
$$
In this case, for all $\beta\in\cplus$, all statements in Thm \ref{thm:bibo} hold.

The operator $\SmallSysNode$ is a scattering passive system node if and only if
$$
	N:=\bbm{H_0&A_1^*&0\\A_1&2R-H_0-EE^* &-E
		\\0&-E^*&I}\geq0.
$$

Assume $H(\cdot)>0$ and $H(\cdot),\dot H(\cdot),H(\cdot)^{-1}\in C^1(\rplus;\Zscr)$ strongly. For all $w_0\in H^1(-\tau,0;\Zscr)$, $e\in H^1_{loc}(\rplus;\Escr)$, there are unique $z$ and $f$, such that $(e,z,f)$ solves \eqref{eq:delayprel} classically.
\end{theorem}

We point out that $M\geq0$ is possible only if $0<H_0\leq2R$, i.e., if there is delay in the system, then we can have a Dirac node only if there is also some internal damping.

\begin{proof}
We first prove that if $M\geq0$ then $\SmallSysNode$ is a Dirac node.  It is clear that \eqref{eq:InjBeta} is injective for $\beta=1$, so it only remains to verify that $\sbm{\AB\\-\CD}$ is maximal dissipative. For all $\sbm{w\\z}\in\dom A$ and $e\in\Escr$, we indeed have
\begin{equation}\label{eq:DnodeDiss}
\begin{aligned}
2\re\Ipdp{\bbm{\AB\\-\CD}\bbm{w\\z\\e}}{\bbm{w\\z\\e}}_{\sbm{\Xscr\\\Escr}} &= \\
2\re\int_{-\tau}^0 \Ipdp{H_0w(\xi)}{w'(\xi)}\ud\xi 
	-2\re\Ipdp{E^*z}{e}&\\
	+2\re\Ipdp{(J-R)z-A_1w(-\tau)+Ee}{z} &=\\
	-\Ipdp{M\bbm{w(-\tau)\\z}}{\bbm{w(-\tau)\\z}}&,
\end{aligned}
\end{equation}
which establishes dissipativity if $M\geq0$.

Next, let $\sbm{v\\x}\in\Xscr$ and $f\in\Escr$ be arbitrary. We need to find $\lambda\in\cplus$, $\sbm{w\\z}\in\dom A$ and $e\in\Escr$, such that
\begin{equation}\label{eq:maxdiss}
	\left(\lambda I-\bbm{\AB\\-\CD}\right)\bbm{w\\z\\e}=\bbm{v\\x\\f}.
\end{equation}
Straightforward calculations show that \eqref{eq:maxdiss} is equivalent to
\begin{equation}\label{eq:wsol}
	w(\xi)=e^{\lambda\xi}w(0)-\int_{0}^\xi e^{\lambda(\xi-\eta)}v(\eta)\ud\eta,
\end{equation}
\begin{equation}\label{eq:esol}
e=f/\lambda-E^*w(0)/\lambda\qquad\text{and}
\end{equation}
\begin{equation}\label{eq:w0eq}
\begin{aligned}
\left(A_1e^{-\lambda\tau}+\lambda I-J+R+EE^*/\lambda\right)w(0) &= \\x+Ef/\lambda
-A_1\int^0_{-\tau} e^{-\lambda(\tau+\eta)}v(\eta)\ud\eta.&
\end{aligned}
\end{equation}
Due to the boundedness of the operators, for $\lambda\in\rplus$ large enough, \eqref{eq:w0eq} can be solved for $w(0)$. Defining $w$ by \eqref{eq:wsol} and setting $z:=w(\tau)$, we then get $\sbm{w\\z}\in\dom A$. Defining $e$ by \eqref{eq:esol}, we have that $\SmallSysNode$ is a Dirac node if $M\geq0$.

Now we prove the converse, so let $\omega,z\in\Zscr$ be such that 
$$
	\Ipdp{M\bbm{\omega\\z}}{\bbm{\omega\\z}}<0.
$$
By the definition of $\dom A$ and \eqref{eq:DnodeDiss}, in order to prove that $\sbm{\AB\\-\CD}$ is not dissipative, all we need to do is to find a $w\in H^1(-\tau,0;\Zscr)$, such that $w(-\tau)=\omega$ and $w(0)=z$. Take for instance $w(\xi):=z-\xi/\tau\cdot(\omega-z)$, $0\leq \xi\leq\tau$.

Adding $2\re\Ipdp{e}{E^*z}-\|e\|^2+\|E^*z\|^2$ to \eqref{eq:DnodeDiss}, we get that $\SmallSysNode$ is scattering dissipative if and only if $N\geq0$ with an argument almost identical to the above.

Finally, the smoothness assumptions on $H(\cdot)$ imply that \eqref{eq:standing2} holds with $G$ zero, so that Thm \ref{thm:bibo} is applicable. Now define $z_0:=H(0)^{-1}w_0(0)$ to get $P(0)\sbm{w_0\\z_0}\in\dom A$. By Thm \ref{thm:bibo}, a classical solution $(e,x,f)$ with $x(0)=\sbm{(Hz)|_{[-\tau,0]}\\z(0)}=\sbm{w_0\\z_0}$ exists, and it is unique by Thm \ref{prop:Power}.
\end{proof}

Curiously, scattering passivity implies impedance passivity: $N\geq0\implies M\geq0$. We get an example where $M,N\geq0$ in Thm \ref{thm:DiracNode} by choosing $H_0:=E:=1$, $A_1:=10$ and $R:=100$. It seems like $N\geq0$ only occurs when there is rather strong damping, i.e., when $R$ is large.

\section{Conclusions}

The paper exhibits how time-varying port-Hamiltonian systems, which are not necessarily of boundary-control type, can be solved using the theory of time-varying well-posed linear systems, which was presented in \cite{Kur20}. The external Cayley transform is the tool for passing between the two system classes, and the theory is illustrated with a few examples. 

Currently, the theory of evolution families on Hilbert spaces is limited, and it would need some further development in order to get much farther in the direction of the present paper. Then progress could be made on time-varying port-Hamiltonian systems, using techniques from \cite{JaLa21} and \cite{Kur20}.

\section{Acknowledgment}

The author thanks Benjamin Unger for the example in \S\ref{sec:Ex}. The first part of Thm \ref{thm:DiracNode} is due to Unger and his coauthors.

\def\cprime{$'$}

\end{document}